\documentclass[11pt]{article}

\usepackage[margin=1in]{geometry}
\usepackage{amsmath,amssymb,amsthm,mathtools}
\usepackage[colorlinks=true,linkcolor=blue,citecolor=blue,urlcolor=blue]{hyperref}

\newtheorem{theorem}{Theorem}[section]
\newtheorem{proposition}[theorem]{Proposition}
\newtheorem{lemma}[theorem]{Lemma}

\theoremstyle{remark}
\newtheorem{remark}[theorem]{Remark}

\newcommand{\R}{\mathbb R}
\newcommand{\C}{\mathbb C}
\newcommand{\Q}{\mathbb Q}
\newcommand{\Z}{\mathbb Z}
\newcommand{\OK}{\mathcal O_K}
\newcommand{\vol}{\operatorname{vol}}
\newcommand{\covol}{\operatorname{covol}}
\newcommand{\Sym}{\operatorname{Sym}}
\newcommand{\Herm}{\operatorname{Herm}}
\newcommand{\Tr}{\operatorname{Tr}}
\newcommand{\Id}{\operatorname{Id}}
\newcommand{\E}{\mathbb E}
\newcommand{\Prob}{\mathbb P}
\newcommand{\HS}{\operatorname{HS}}
\newcommand{\op}{\operatorname{op}}
\newcommand{\calE}{\mathcal E}

\title{Stochastically evolving ellipsoids with symmetries}
\author{Elisha B. Abuya \and Nihar Gargava \and Yufei Zhao}
\date{June 3, 2026}

\begin{document}
\maketitle

\begin{abstract}
We prove that there is a universal constant $c>0$ such that, along an infinite
sequence of dimensions $N$, there are lattice sphere packings in $\mathbb R^N$
of density at least $cN^2\log\log N\,2^{-N}$, improving the previous best bound
due to Klartag by a $\log\log N$ factor.  The proof follows Klartag's stochastic
ellipsoid evolution process, subject to the cyclotomic symmetries
introduced by Venkatesh.
\end{abstract}

\paragraph*{Statement on AI use.}

The first two authors developed an approach combining
Venkatesh's~\cite{Venkatesh} cyclotomic symmetries with
Klartag's~\cite{Klartag} stochastic ellipsoid method.
They focused on the case
of modules of rank two over the ring of cyclotomic integers like in Venkatesh's setup.
This seemed to obtain a suboptimal lower bound on sphere packings which is in between 
Venkatesh's lower bound and Klartag's lower bound.
Before moving to high ranks, the authors struggled to conclude if this is a miscalculation or a methodological problem
in the setup.

Independently, the third author, inspired by the recent OpenAI announcement
of an AI-generated disproof of the Erd\H{o}s unit distance
conjecture~\cite{ABGLSSTWW, OpenAI}, prompted the
GPT-5.5 Pro model to extend Venkatesh's construction to obtain improved lower
bounds for high-dimensional sphere packings.  The model proposed combining
Venkatesh's cyclotomic symmetries with Klartag's stochastic ellipsoid method
and, after further prompting to continue pushing on the approach, produced an
improvement over the previous best bound.  The model discovered the idea of letting the rank grow and successfully
proved the bound given below.

The third author then contacted the first author, and the three authors
joined forces to verify the proof and edit the writing, with ChatGPT and Codex
used extensively in the process.  The ancillary files accompanying this paper
include a complete log of the ChatGPT conversation that led to the proof.

The first author is choosing to use a pseudonym.

\section{Introduction}

A lattice sphere packing in $\mathbb R^N$ is a collection of disjoint
congruent Euclidean balls whose centers form a full-rank lattice.  Its density
is the proportion of space covered by the balls.  Let $\Delta_N^{\rm lat}$
denote the supremum of the densities of all lattice sphere packings in
$\mathbb R^N$.

The Minkowski--Hlawka theorem~\cite{Hlawka} gives the lattice-packing lower
bound $\Delta_N^{\rm lat}\ge (2+o(1))\,2^{-N}$.
Rogers~\cite{Rogers} proved the first linear improvement to this lower bound,
$\Delta_N^{\rm lat}\ge cN\,2^{-N}$.  Schmidt~\cite{Schmidt} later gave another
proof of a linear lower bound for lattice packings.  Subsequent work of
Davenport and Rogers~\cite{DavenportRogers}, Ball~\cite{Ball}, and
Vance~\cite{Vance} improved the constant in this lower bound and obtained
further gains in special dimensions.  Using lattices with
cyclotomic symmetries, Venkatesh~\cite{Venkatesh} obtained a
lattice-packing lower bound with an
additional factor of order $\log\log N$ along a sparse sequence of
dimensions; Gargava and Viazovska~\cite{GargavaViazovska} later improved its
leading constant.  For unrestricted sphere packings,
Krivelevich, Litsyn, and Vardy~\cite{KLV} used
graph-theoretic arguments to prove a lower bound of order $N2^{-N}$, and
Campos, Jenssen,
Michelen, and Sahasrabudhe~\cite{CJMS} recently improved this
unrestricted-packing lower
bound to
$(1/2-o(1))N\log N\,2^{-N}$.
Klartag~\cite{Klartag} recently obtained the lattice-packing lower bound
$\Delta_N^{\rm lat}\ge cN^2\,2^{-N}$ in every dimension by introducing a
stochastically growing ellipsoid process.  The main
idea of this paper is to combine Klartag's process with the cyclotomic
symmetries in Venkatesh's construction.

\begin{theorem}\label{thm:main}
There exists a universal constant $c>0$ and an infinite
sequence of dimensions $N$ for which
\[
        \Delta_N^{\rm lat}\ge cN^2\log\log N\,2^{-N}.
\]
\end{theorem}

On various occasions, Venkatesh has conjectured the upper bound
$\Delta_N^{\rm lat}\le N^2(\log N)^{O(1)}\,2^{-N}$.  This is motivated by the
fact that lattices in $\mathbb R^N$ have on the order of $N^2$ degrees of
freedom.

For comparison, here are the known upper bounds for arbitrary sphere packings.
Earlier bounds are due to Blichfeldt~\cite{Blichfeldt},
Rankin~\cite{Rankin}, Rogers~\cite{RogersPacking}, and
Levenshtein~\cite{Levenshtein}.
Kabatiansky and Levenshtein~\cite{KabatianskyLevenshtein}
proved that every sphere packing in $\mathbb R^N$ has density at most
$2^{-(0.5990\ldots+o(1))N}$.
This remains the best known exponential rate; later work of Cohn and
Zhao~\cite{CohnZhao}, Sardari and Zargar~\cite{SardariZargar}, and
Zargar~\cite{Zargar} gives constant-factor improvements.
The Cohn--Elkies~\cite{CohnElkies} linear program is another central framework
for upper-bounding sphere-packing densities; it recovers the
Kabatiansky--Levenshtein bound~\cite{CohnZhao}.
In the opposite direction, Edwin~\cite{Edwin} recently proved
that the optimal Cohn--Elkies bound is at least
$\frac14(e/8)^{N/2}=2^{-(0.77865\ldots)N-2}$, matching a statistical-mechanical
prediction by Torquato and Stillinger~\cite{TorquatoStillinger}.  This is a
limitation
on the linear-programming method, rather than a known packing construction.
For fixed dimensions, the optimal sphere-packing density is known exactly only
in dimensions $1,2,3,8,$ and $24$; the one-dimensional case is immediate, and
the remaining cases were established by Fejes T\'oth~\cite{FejesToth},
Hales~\cite{Hales}, Viazovska~\cite{Viazovska}, and Cohn, Kumar, Miller,
Radchenko, and Viazovska~\cite{CKMRV}.  By contrast, Cohn and
Kumar~\cite{CohnKumar} noted that the densest lattices are known in every
dimension up to $8$ and also in dimension $24$, and conjectured that lattice
packings are suboptimal in all sufficiently high dimensions.

Our proof follows the structure of Klartag's~\cite{Klartag} paper quite
closely, with additional modifications needed to make the cyclotomic symmetries
work.  We advise the reader to read it first, as the same basic mechanism
appears there without the complications introduced by the symmetries.

Klartag's~\cite{Klartag} bound $\Delta_N^{\rm lat}\ge cN^2\,2^{-N}$
is obtained by choosing a random lattice, starting with a Euclidean ball
containing no nonzero points of that lattice, and letting it evolve
stochastically through ellipsoids.  Whenever the boundary of the evolving
ellipsoid reaches a nonzero lattice point, the future motion is constrained to
keep that point on the boundary.  Here is some intuition for this process.  One
may heuristically regard the lattice points as a Poisson process of intensity
one, so that each constant amount of newly explored volume should reveal a new
lattice point.
Each new contact imposes only one linear constraint on the quadratic form
defining the ellipsoid, while an ellipsoid has on the order of $N^2$ degrees
of freedom.  The process should therefore have room to accumulate on the order
of $N^2$ contact points before it freezes, suggesting that it can reach volume
on the order of $N^2$ while its interior remains free of nonzero lattice
points.  The ellipsoid does not grow monotonically: it expands and contracts
as it moves.  The key point is that the concavity of $\log\det$ produces a
favorable determinant drift whose magnitude is governed by the number of
degrees of freedom that remain after the accumulated contact constraints are
imposed.  Controlling this drift makes the heuristic rigorous and yields the
quadratic gain.

An earlier construction of Venkatesh~\cite{Venkatesh} supplies a complementary
idea, which we combine here with Klartag's process.
Related orbit-saving ideas appear
in work of Gargava~\cite{Gargava}, Gargava, Serban and Viazovska~\cite{GSV}, and
Gargava and
Viazovska~\cite{GargavaViazovska}.
Let us explain below how the cyclotomic symmetries interact with the stochastic ellipsoids.

The
cyclotomic module lattices carry a symmetry group of order $m$, and every
nonzero lattice point belongs to an orbit of $m$ points.  If the ellipsoid is
required to preserve this symmetry, then reaching one lattice point means
reaching its entire orbit.  More importantly, keeping that whole orbit on the
boundary imposes only the same single condition as keeping one of its points
there.
Thus each new contact gives $m$ lattice points for the price of one.  Taking
the module rank $r=\lceil(\log\varphi(m))^4\rceil$, the process allows the
ellipsoid to grow in volume by a factor of at least $cm\varphi(m)r^2$.  The
ambient dimension is $N=\varphi(m)r$.  Choosing $m$ to be primorial as in
Venkatesh's construction~\cite{Venkatesh} gives
$m\asymp\varphi(m)\log\log N$, which yields the additional $\log\log N$
factor in our bound.

After the appropriate cyclotomic symmetries and module lattices have been set
up, the proof largely follows Klartag's argument.  Restricting the motion to
cyclotomic-invariant ellipsoids leaves
less randomness than in the original process.  Venkatesh's construction uses
modules of rank two, whereas here we let the module rank grow to compensate
for this loss.  See Section~\ref{subsec:angular-block-mass} for this step,
which is a new feature of this proof compared to earlier works.

Throughout the paper $c,C,\ldots>0$ denote universal constants,
whose values may change from line to line unless explicitly fixed.  Subscripted
constants remain fixed once introduced, and all implicit constants below are
universal unless stated otherwise.  As the relevant parameter tends to
infinity, we write $f=O(g)$ if $|f|\le Cg$, $f=o(g)$ if $f/g\to0$, and
$f\sim g$ if $f/g\to1$.  For positive quantities $f$ and $g$, we write
$f\asymp g$ if $cg\le f\le Cg$.

\begin{remark}[Relation to lattice-based cryptography]

It is a great coincidence that the cyclotomic module lattices in our packing
construction are closely related to the lattices used widely in
lattice-based cryptography.  In August 2024, the U.S.\ National Institute of
Standards and Technology approved its first three standards for post-quantum
cryptography; two are module-lattice-based: ML-KEM, a key-encapsulation
mechanism derived from CRYSTALS-Kyber~\cite{BDKLLSSSS}, and ML-DSA, a
digital-signature scheme derived from CRYSTALS-Dilithium~\cite{DKLLSSS}.  The
security of ML-KEM is based on Module Learning With Errors, while that of
ML-DSA is based on Module Learning With Errors and a variant of Module Short
Integer Solution called SelfTargetMSIS~\cite{Nat24a,Nat24b}.  Here the ``ML''
prefix stands for module-lattice.  Motivated by cryptographic applications,
\cite{BPTW} studies a related discretization of the natural invariant
measure on random module lattices of fixed rank, analogous to the Haar
probability measure considered in our paper.

In the standardized parameter sets for both schemes, the underlying cyclotomic
polynomial has degree $256$, while the relevant module dimensions are fixed and
at most $8$~\cite{Nat24a,Nat24b}.  Questions have been posed about the effects
of the module rank and cyclotomic degree on lattice
reduction~\cite{DEP}.  Our work indicates that the short vectors seem
to prefer certain directions more than others.  This effect is more prominent
when the rank is small (cf.\ Remark~\ref{re:effect_of_rank}).  This is
reminiscent of the skewness gap studied in~\cite[\S 4.4]{DEP}, which
measures imbalance among the archimedean embeddings of short vectors.
\end{remark}

\section{Cyclotomic lattices}\label{sec:cyclotomic}

\subsection{Cyclotomic lattices and the mean-value formula}

We begin with the arithmetic parameters used throughout the paper.  Let $m>2$
be an integer, let $\zeta_m$ be a primitive $m$th root of unity, and set
$K=\Q(\zeta_m)$.
Thus $K$ is the $m$th cyclotomic field, of degree $\varphi(m)$ over $\Q$,
where $\varphi$ is Euler's totient function.  Its group of $m$th roots of
unity is
\[
        \mu_m=\{\text{$m$th roots of unity}\}\subset K^\times.
\]
This will be the symmetry group of our construction.

We take the growing module rank
\[
        r=\lceil(\log\varphi(m))^4\rceil
\]
and work with the free rank-$r$ $K$-module $V=K^r$.  Its realification
$V_\R=V\otimes_\Q\R$ is a real vector space of dimension
\[
        N=\dim_\R V_\R=\varphi(m)r.
\]
We write $B_N=\{x\in\R^N:\|x\|<1\}$ for the open Euclidean ball of radius
one.

Set $K_\infty=K\otimes_\Q\R$, the archimedean real algebra obtained from $K$
by extending scalars from $\Q$ to $\R$.  The field $K$ is totally imaginary,
so its complex embeddings come in conjugate pairs.  Let
\[
        s=\frac{\varphi(m)}2,
\]
and choose embeddings $\sigma_1,\dots,\sigma_s$, one from each conjugate pair.
They give the decompositions
\[
        K_\infty\cong\prod_{j=1}^s\C,
        \qquad\text{and}\qquad
        V_\R\cong\bigoplus_{j=1}^s\C^r.
\]
We write an element $x\in V_\R$ as $x=(x_1,\dots,x_s)$, where each
$x_j\in\C^r$.  Write $\OK$ for the ring of algebraic integers in $K$.  This
ring plays the role of $\Z$ inside $K$, and $\OK^r$ plays the role of the
standard lattice $\Z^r$.
Equip each factor $\C^r$ with the same scalar multiple of its standard
Euclidean structure, chosen so that the resulting product Euclidean structure
on $V_\R$ satisfies
\[
        \covol(\OK^r)=\vol(B_N).
\]
All norms, Hilbert--Schmidt norms, volumes, and Lebesgue measures below refer
to this structure.  Thus
\[
        \|x\|^2=\sum_{j=1}^s\|x_j\|^2.
\]
This scalar rescaling does not change the block proportions, the orthogonal
decompositions used below, or the orthogonality of the cyclotomic action.

The cyclic group $\mu_m\subset K^\times$ has order $m$.  It acts on $K^r$ by
coordinatewise scalar multiplication:
\[
        \zeta\cdot(y_1,\dots,y_r)
        =(\zeta y_1,\dots,\zeta y_r).
\]
Extending this action $\R$-linearly to $V_\R=K^r\otimes_\Q\R$, and using the
block decomposition above, gives
\[
        \zeta\cdot(x_1,\dots,x_s)
        =\bigl(\sigma_1(\zeta)x_1,\dots,\sigma_s(\zeta)x_s\bigr).
\]
The action is orthogonal for the rescaled Euclidean structure, since
$|\sigma(\zeta)|=1$ for every complex embedding $\sigma$ and every
$\zeta\in\mu_m$.  It is free away from the origin: if $h\in\mu_m$ and $h\ne1$,
then $h-1$ is nonzero in the field $K$, hence multiplication by $h-1$ is
invertible on $K^r$ and on $V_\R$.  Thus $hx=x$ implies $x=0$, and every
nonzero $\mu_m$-orbit has exactly $m$ points.

The product decomposition of $K_\infty$ gives
\[
        \operatorname{SL}_r(K_\infty)
        \cong\prod_{j=1}^s\operatorname{SL}_r(\C).
\]
An element $g\in\operatorname{SL}_r(K_\infty)$ can therefore be written as a
tuple
\[
        g=(g_1,\dots,g_s),
\]
where each $g_j\in\operatorname{SL}_r(\C)$,
and its action on $V_\R\cong\bigoplus_{j=1}^s\C^r$ is simply
\[
        g\cdot(x_1,\dots,x_s)=(g_1x_1,\dots,g_sx_s).
\]
Thus $\operatorname{SL}_r(K_\infty)$ is the group of determinant-one
$K_\infty$-linear changes of coordinates.  Applying such a change of
coordinates to the standard module $\OK^r$ gives a lattice
\[
        \Lambda_g=g\OK^r\subset V_\R.
\]

A matrix $\gamma\in\operatorname{SL}_r(\OK)$ merely changes the $\OK$-basis of
$\OK^r$, so $\gamma\OK^r=\OK^r$.  Consequently,
\[
        \Lambda_{g\gamma}=g\gamma\OK^r=g\OK^r=\Lambda_g.
\]
Thus the map $g\mapsto\Lambda_g$ factors through the left cosets
$g\operatorname{SL}_r(\OK)$, and the parameter space for these lattices is
the quotient
\[
        \operatorname{SL}_r(K_\infty)/\operatorname{SL}_r(\OK).
\]
Each coset determines a lattice $\Lambda_g$.  By the
Borel--Harish-Chandra~\cite{BorelHarishChandra} theorem, this quotient has
finite $\operatorname{SL}_r(K_\infty)$-invariant volume; we normalize Haar
measure to make it a probability measure, denoted $\mu$.

Since $g$ is $K_\infty$-linear, it commutes with scalar multiplication by
$\mu_m$, so $\Lambda_g$ is $\mu_m$-stable.  Moreover,
\[
        \det\nolimits_\R(g)=\prod_{j=1}^s|\det\nolimits_\C(g_j)|^2=1.
\]
Thus $g$ preserves covolume, and hence
\[
        \covol(\Lambda_g)=\covol(\OK^r)=\vol(B_N).
\]

For later use we record the number-field Siegel mean formula.  Related work on
number-field module lattices
includes higher-moment formulas
of Gargava, Serban, and Viazovska~\cite{GSV}, as well
as shortest-vector estimates of Gargava, Serban, Viazovska, and
Viglino~\cite{GSVV}.  The formula below appears
in the rank-two setting in Venkatesh~\cite{Venkatesh} and in the
division-algebra setting in Gargava~\cite[Theorem~16]{Gargava}.  The
form needed here also follows by taking the first-moment case in
Kim~\cite[the corollary following Theorem~1.2 and subsequent discussion]{Kim}.
Kim first states the corresponding corollary under a class-number-one
assumption and then explains that the same formula holds on any fixed
Steinitz-class component.  We apply this extension to the component containing
the free module $\OK^r$.

\begin{theorem}\label{thm:siegel}
Let $r\ge2$.  Recall that $\Lambda_g=g\OK^r$, and let $\mu$ denote the
normalized $\operatorname{SL}_r(K_\infty)$-invariant probability measure
on $\operatorname{SL}_r(K_\infty)/\operatorname{SL}_r(\OK)$.  Then, for every
nonnegative Borel function $f:V_\R\to[0,\infty]$,
\[
        \int_{\operatorname{SL}_r(K_\infty)/\operatorname{SL}_r(\OK)}
        \sum_{0\ne v\in\Lambda_g}f(v)\,d\mu(g)
        =\frac1{\covol(\OK^r)}\int_{V_\R}f(x)\,dx.
\]
Here $dx$ is the Lebesgue measure for the Euclidean structure fixed above.
\end{theorem}

\subsection{Invariant quadratic forms}

We use a space of quadratic forms compatible with the cyclotomic structure.  Here
$\Sym V_\R$ denotes the real vector space of symmetric bilinear forms on
$V_\R$, identified with self-adjoint matrices using the Euclidean structure
fixed above.  Set
\[
        \mathcal Q_K
        =\{A\in\Sym V_\R:A\text{ is }K_\infty\text{-linear}\}.
\]
Under the decomposition $V_\R\cong\bigoplus_{j=1}^s\C^r$, this is the space
of block-diagonal self-adjoint operators
\[
        \mathcal Q_K\cong\bigoplus_{j=1}^s\Herm(\C^r),
\]
where $\Herm(\C^r)$ denotes the real vector space of Hermitian $r\times r$
matrices.  In particular,
\[
        \dim_\R\mathcal Q_K
        =sr^2=\frac{\varphi(m)r^2}{2}=\frac{Nr}{2}.
\]
Every $A\in\mathcal Q_K$ is $\mu_m$-invariant.  Indeed, $A$ commutes with
scalar multiplication by $\zeta\in\mu_m$, and this scalar multiplication is
orthogonal.  Hence
\[
        A[\zeta x]
        =\langle A\zeta x,\zeta x\rangle
        =\langle\zeta Ax,\zeta x\rangle
        =\langle Ax,x\rangle
        =A[x].
\]
Thus every positive-definite $A\in\mathcal Q_K$ defines a $\mu_m$-stable
ellipsoid.  Also, $\Id\in\mathcal Q_K$, as required for the scalar initial
matrix in the stochastic construction.  We equip $\Sym V_\R$ with the
Hilbert--Schmidt inner product
\[
        \langle A,B\rangle_{\HS}=\Tr(AB).
\]

The following arithmetic lattice theorem is the main technical result of the
paper.  After showing that it implies Theorem~\ref{thm:main}, the remainder of
the paper is devoted to its proof.

\begin{theorem}[Arithmetic lattice ellipsoid]\label{thm:arithmetic-lattice}
There is a universal constant $c>0$ such that,
for all sufficiently large integers $m$, there exist a $\mu_m$-stable lattice
$\Lambda\subset V_\R$ of covolume $\vol(B_N)$ and an ellipsoid $E\subset V_\R$
such that
\[
        E\cap\Lambda=\{0\}
        \qquad\text{and}\qquad
        \vol(E)\ge cmNr\vol(B_N).
\]
\end{theorem}

We explain how this theorem implies the packing bound stated in the
introduction.

\begin{proof}[Proof of Theorem~\ref{thm:main}]
Now specialize to
\[
        m=\prod_{p\le x}p
\]
as $x\to\infty$.  By Theorem~\ref{thm:arithmetic-lattice}, and since
$N=\varphi(m)r$,
\[
        \frac{\vol(E)}{\covol(\Lambda)}
        \ge cmNr
        =cN^2\frac{m}{\varphi(m)}.
\]
The Mertens product theorem and the prime number theorem give
\[
        \frac{m}{\varphi(m)}
        \asymp\log x
        \asymp\log\log m.
\]
Moreover,
$N=\varphi(m)\lceil(\log\varphi(m))^4\rceil$ and
$m/\varphi(m)=O(\log x)$ imply that
$\log N\sim\log m$.  Hence $\log\log N\sim\log\log m$.

The translates of $E/2$ by $\Lambda$ are disjoint.  After an invertible linear
change of variables, they form a lattice sphere packing of density
\[
        \frac{\vol(E/2)}{\covol(\Lambda)}
        \ge cN^2\log\log N\,2^{-N}.
\]
This proves the result along the infinite sequence of dimensions obtained by
this choice of $m$.
\end{proof}

It remains to prove Theorem~\ref{thm:arithmetic-lattice}.  We next record the
pointwise projection formula used in the stochastic estimates.

For $x\ne0$, write
\[
        x=(x_1,\dots,x_s)\in\bigoplus_{j=1}^s\C^r.
\]
For each block, put
\[
        p_j=\frac{\|x_j\|^2}{\|x\|^2}.
\]
Thus $\sum_jp_j=1$.  Let
\[
        \Pi_K:\Sym V_\R\to\mathcal Q_K
\]
denote Hilbert--Schmidt orthogonal projection.

\begin{lemma}[Projected rank-one norm]\label{lem:projection}
For every nonzero $x\in V_\R$,
\[
        \left\|\Pi_K(x\otimes x)\right\|_{\HS}^2
        =\frac12\sum_{j=1}^s\|x_j\|^4.
\]
Consequently, if
\[
        \beta(x)^2=
        \frac{\left\|\Pi_K(x\otimes x)\right\|_{\HS}^2}{\|x\|^4},
\]
then
\[
        \beta(x)^2=\frac12\sum_{j=1}^s p_j^2.
\]
Consequently,
\[
        \frac1{2s}\le \beta(x)^2\le\frac12.
\]
In particular, $\beta(x)>0$ for every nonzero $x$.
\end{lemma}

\begin{proof}
On one block $\C^r$, the orthogonal projection of $u\otimes u$ onto the
Hermitian operators is
\[
        \frac12\bigl(u\otimes u+(iu)\otimes(iu)\bigr).
\]
Indeed, the displayed expression is the average of $u\otimes u$ under the
order-four group generated by multiplication by $i$.  Since $u\perp iu$,
\[
        \left\|\frac12\bigl(u\otimes u+(iu)\otimes(iu)\bigr)\right\|_{\HS}^2
        =\frac14\|u\otimes u\|_{\HS}^2+
         \frac14\|(iu)\otimes(iu)\|_{\HS}^2
        =\frac12\|u\|^4.
\]
Since every operator in $\mathcal Q_K$ is block diagonal, the off-diagonal
blocks of $x\otimes x$ are orthogonal to $\mathcal Q_K$ and hence vanish under
$\Pi_K$.
Summing over the $s$ diagonal blocks gives the claimed identity.  The bounds
on $\beta(x)^2$ follow from
\[
        \frac1s\le\sum_{j=1}^sp_j^2\le1.\qedhere
\]
\end{proof}

\section{Stochastic ellipsoid evolution}

We continue to use the global notation and Euclidean structure fixed in
Section~\ref{sec:cyclotomic}.  This section records
the subspace version of the stochastic estimates used later.  The pathwise
construction is the contact-preserving process from~\cite[Proposition~2.3 and
Corollary~2.4]{Klartag}, with the full Euclidean space of symmetric matrices
replaced by $\mathcal Q_K$.  The modifications used here are that
the Brownian motion is constrained inside this space, and the loss of Brownian
directions is counted by $\mu_m$-orbits of contact points rather than by
individual contact points.

In this section $\Lambda\subset V_\R$ is a fixed $\mu_m$-stable lattice of covolume $\vol(B_N)$.  For a positive-definite $A\in\mathcal Q_K$ define the ellipsoid
\[
        \calE_A=\{x\in V_\R:A[x]<1\},
\]
where $A[x]=\langle Ax,x\rangle$.
We say that $A$ is $\Lambda$-free if
\[
        \calE_A\cap\Lambda=\{0\}.
\]

\begin{lemma}\label{lem:standard-geometric}
Let $A\in\Sym V_\R$ be positive definite and $\Lambda$-free.
\begin{itemize}
\item[(i)] \cite[Appendix~A, Lemma~A.1]{Klartag}
$|\partial\calE_A\cap\Lambda|\le 2(2^N-1)$.
\item[(ii)] (Minkowski's first theorem)
$\vol(\calE_A)\le 2^N\covol(\Lambda)=2^N\vol(B_N)$.
Equivalently, in our normalization, $\det A\ge4^{-N}$.
\end{itemize}
\end{lemma}

Following the notation in~\cite[equation~(13)]{Klartag}, for
$A\in\mathcal Q_K$ define the contact-preserving, or active, direction
space
\begin{equation}
        F_A=\{B\in\mathcal Q_K:B[x]=0
        \text{ for every }x\in\partial\calE_A\cap\Lambda\}.
\label{eq:FA}
\end{equation}
Thus $B\in F_A$ precisely when additive perturbations in the direction $B$
keep every current contact point on the boundary.
Write $\Pi_A:\Sym V_\R\to F_A$ for Hilbert--Schmidt orthogonal projection.
This definition and the arguments below are essentially the same as
in~\cite{Klartag}, except that we restrict throughout to $\mathcal Q_K$.

\begin{proposition}[Invariant-subspace stochastic process]\label{prop:process}
Let $a_0>0$ be such that $a_0\Id$ is $\Lambda$-free, and let $(W_t)_{t\ge0}$
be a Brownian motion in $\mathcal Q_K$.  There exists a continuous
process $(A_t)_{t\ge0}$ with values in $\mathcal Q_K$, adapted to the
filtration associated with $(W_t)_{t\ge0}$ and starting at $A_0=a_0\Id$, with
the following properties.
\begin{itemize}
\item[(i)] For every $t\ge0$, the matrix $A_t$ is positive definite and $\Lambda$-free.
\item[(ii)] Contact points are preserved: if $x\in\partial\calE_{A_t}\cap\Lambda$ and $s\ge t$, then $x\in\partial\calE_{A_s}\cap\Lambda$.
\item[(iii)] The process satisfies $dA_t=\Pi_{A_t}(dW_t)$ until the first time
at which $F_{A_t}=\{0\}$, after which it is constant.
\end{itemize}
\end{proposition}

\begin{proof}
We apply the construction of~\cite[Lemmas~2.1 and~2.2,
Proposition~2.3, and Corollary~2.4]{Klartag}, replacing the full
symmetric-matrix space by the fixed Euclidean subspace $\mathcal Q_K$.  This
space contains $\Id$, the active space $F_A$ in~\eqref{eq:FA} is the
intersection with $\mathcal Q_K$ of the space in~\cite{Klartag}, and
Lemma~\ref{lem:standard-geometric} supplies the required bounds on lattice-free
ellipsoids.  The measurability and stopping-time arguments are unchanged,
since $\mathcal Q_K$ is fixed and finite dimensional and $\Lambda$ is
countable.

It remains only to check the Brownian exit argument in
\cite[Lemma~2.2]{Klartag}.  If $F_A\ne\{0\}$, choose $0\ne B\in F_A$ and then
$x\in V_\R$ with $B[x]\ne0$.  Since
\[
        \langle B,x\otimes x\rangle_{\HS}=B[x]\ne0,
\]
the projection $\Pi_A(x\otimes x)$ is nonzero.  Hence the corresponding
scalar component of Brownian motion in $F_A$ has $\liminf=-\infty$ almost
surely, as required.  Klartag's recursive construction therefore gives a
continuous, adapted, positive-definite, $\Lambda$-free, contact-preserving
process driven by Brownian motion $W_t$ in $\mathcal Q_K$.  It has only finitely many
phases, since each nontrivial phase adds a contact point and
Lemma~\ref{lem:standard-geometric} bounds the total number of contacts.
The It\^o formulation in~\cite[Corollary~2.4]{Klartag} gives
$dA_t=\Pi_{A_t}(dW_t)$, completing the proof.
\end{proof}

\begin{lemma}[Dimension loss by orbit contacts]\label{lem:Fdimension}
For every $t$,
\[
        \dim F_{A_t}\ge
        \frac{Nr}{2}-\frac{|\partial\calE_{A_t}\cap\Lambda|}{m}.
\]
\end{lemma}

\begin{proof}
Each contact point $x$ imposes the linear condition $B[x]=0$ on
$B\in\mathcal Q_K$.  If $h\in \mu_m$, then for all $B\in\mathcal Q_K$,
\[
        B[hx]=B[x],
\]
because $B$ is $\mu_m$-invariant and $h$ is orthogonal.  Thus all points in the same $\mu_m$-orbit impose the same linear condition.

The action is free away from the origin, as observed in
Section~\ref{sec:cyclotomic}, so every contact orbit has exactly $m$ points.
The span of the contact conditions therefore has dimension at most
$|\partial\calE_{A_t}\cap\Lambda|/m$.  Since $F_{A_t}$ is their common kernel
in the $(Nr/2)$-dimensional space $\mathcal Q_K$, the claimed bound follows.
\end{proof}

In the full-space process of~\cite{Klartag}, each independent contact equation removes one
Brownian direction.  Here all points in a single $\mu_m$-orbit impose the same
equation on invariant quadratic forms, so the determinant drift is controlled
by the number of contact orbits rather than by the number of contact vectors.

\subsection{Operator-norm control and determinant drift}

We first need the invariant-subspace analogue of the operator-norm estimate
in~\cite[Lemma~3.1 and Corollary~3.2]{Klartag}.  The same argument applies after
restricting the Brownian motion to $\mathcal Q_K$; we record the
short direct proof.

\begin{lemma}[Operator-norm control]\label{lem:opnorm}
There is a universal constant $C_{\ref{lem:opnorm}}>0$ such that, for every
$t>0$,
\[
        \Prob\left(\|A_t-a_0\Id\|_{\op}>
        C_{\ref{lem:opnorm}}\sqrt{Nt}\right)
        \le e^{-2N}.
\]
More quantitatively, for all $u>0$,
\begin{equation}
        \Prob\left(\|A_t-a_0\Id\|_{\op}>u\right)
        \le 2\cdot 9^N\exp\left(-u^2/(8t)\right).
\label{eq:op-full-tail}
\end{equation}
Moreover, for every $0<q<\infty$ and $0<T<\infty$,
\[
        \sup_{0\le t\le T}\E\|A_t-a_0\Id\|_{\op}^q<\infty.
\]
\end{lemma}

\begin{proof}
For a unit vector $u\in V_\R$, consider the scalar process
$(M_s)_{0\le s\le t}$ defined by
\[
        M_s=u\cdot(A_s-a_0\Id)u.
\]
This is a continuous martingale.  Recall that $\Pi_K$ is orthogonal
projection onto $\mathcal Q_K$ and that $\Pi_{A_\ell}$ is orthogonal
projection onto the active space
$F_{A_\ell}$.  Thus
\[
        [M]_s
        \le \int_0^s\|\Pi_{A_\ell}\Pi_K(u\otimes u)\|_{\HS}^2\,d\ell
        \le s\|u\otimes u\|_{\HS}^2=s.
\]
For completeness, we spell out the standard scalar martingale tail argument.
If $r>0$, then the exponential supermartingale inequality, applied to $M_s$
and to $-M_s$, gives
\[
        \Prob(|M_t|>r)
        \le 2\exp\left(-r^2/(2t)\right),
\]
since $M_0=0$ and $[M]_t\le t$.  Indeed, for every $\lambda>0$,
\[
        \Prob(M_t>r)
        \le \exp\left(-\lambda r+\lambda^2t/2\right),
\]
and the choice $\lambda=r/t$ gives the one-sided bound.

Now let $\mathcal N$ be a $1/4$-net of the unit sphere in $V_\R$ with
$|\mathcal N|\le 9^N$.  For every symmetric operator $B$ on $V_\R$,
\[
        \|B\|_{\op}
        =\sup_{|u|=1}|u\cdot Bu|
        \le 2\max_{u\in\mathcal N}|u\cdot Bu|.
\]
Applying this with $B=A_t-a_0\Id$ and taking a union bound over $\mathcal N$,
we obtain
\[
\begin{aligned}
        \Prob\left(\|A_t-a_0\Id\|_{\op}>u\right)
        &\le
        \sum_{v\in\mathcal N}
        \Prob\left(\left|v\cdot(A_t-a_0\Id)v\right|>\frac{u}{2}\right) \\
        &\le 2\cdot 9^N
        \exp\left(-u^2/(8t)\right),
\end{aligned}
\]
which is~\eqref{eq:op-full-tail}.
Taking the threshold to be $C_{\ref{lem:opnorm}}\sqrt{Nt}$ with
$C_{\ref{lem:opnorm}}$ sufficiently large proves the stated
exponential-in-$N$ bound.  Replacing $t$ by $T$
in~\eqref{eq:op-full-tail} and integrating the resulting tail bound gives the
uniform polynomial moments; the case $t=0$ is trivial.
\end{proof}

For the remaining applications, we choose the initial ellipsoid to be the ball
of radius $1-1/N$.  Thus set
\begin{equation}
        a_0=(1-1/N)^{-2},
\label{eq:a0}
\end{equation}
so that $\calE_{a_0\Id}$ is precisely this ball.

We next prove the determinant estimate needed later.  The It\^o formula for
$\log\det A_t$ is the same calculation as in~\cite[Lemma~3.3]{Klartag}; after
restricting the Brownian motion to $\mathcal Q_K$, the leading drift
is governed by $\dim_\R\mathcal Q_K=Nr/2$.  The orbit-count lemma then
replaces the contact count there by
$|\partial\calE_{A_t}\cap\Lambda|/m$.  The operator-norm error is still
controlled at the ambient vector-dimension scale $N$.

\begin{proposition}[Determinant drift]\label{prop:det-drift}
For all sufficiently large integers $m$ and every
\[
        0<T\le\frac{8\log(mNr)}{Nr},
\]
the process of Proposition~\ref{prop:process} satisfies
\[
        \E\log\det A_T
        \le -\frac{NrT}{4}
        +\frac12\int_0^T\E
        \frac{|\partial\calE_{A_t}\cap\Lambda|}{m}\,dt+O(1).
\]
\end{proposition}
\begin{proof}
View $\Pi_{A_t}$ as the orthogonal projection from $\mathcal Q_K$ onto
$F_{A_t}$.  The same proof as in~\cite{Klartag} applies without change in the
Euclidean space $\mathcal Q_K$.  Indeed, Proposition~\ref{prop:process}
gives $dA_t=\Pi_{A_t}(dW_t)$, Lemma~\ref{lem:standard-geometric} gives the
required determinant lower bound, and Lemma~\ref{lem:opnorm} gives the
required operator-norm moments.  Moreover, $A_t^{-1}\in\mathcal Q_K$ and
$\Tr_{\mathcal Q_K}(\Pi_{A_t})=\dim F_{A_t}$.  In particular, the stochastic
integral has mean zero, $\log\det A_T$ is integrable, and
\begin{equation}
        \E\log\det A_T
        \le N\log a_0
        -\frac12\int_0^T
          \E[\|A_t\|_{\op}^{-2}\dim F_{A_t}]\,dt.
\label{eq:det1}
\end{equation}
Here $N\log a_0$ is bounded by a universal constant, since $a_0=(1-1/N)^{-2}$.

We now follow the good-event estimate in the proof of
\cite[Proposition~3.4]{Klartag}, keeping separate the ambient dimension $N$
and the dimension $Nr/2$ of $\mathcal Q_K$.  Let
\[
        \mathcal G_t=\{\|A_t-a_0\Id\|_{\op}\le
        C_{\ref{lem:opnorm}}\sqrt{Nt}\}.
\]
Recall that $r=\lceil(\log\varphi(m))^4\rceil$.
Since $m\le2\varphi(m)^2$ for every positive integer $m$ and
$N=\varphi(m)r$, we have $\log(mNr)=O(\log\varphi(m))$.  Thus, uniformly
for $0\le t\le T$,
\[
        Nt\le NT\le\frac{8\log(mNr)}r=o(1).
\]
On $\mathcal G_t$, we have
$\|A_t\|_{\op}\le a_0+C_{\ref{lem:opnorm}}\sqrt{Nt}$.
Using $\Prob(\mathcal G_t^c)\le e^{-2N}$ from Lemma~\ref{lem:opnorm},
$a_0=1+O(N^{-1})$, and $0\le\dim F_{A_t}\le Nr/2$, the same good-event argument
gives
\[
        \E[\|A_t\|_{\op}^{-2}\dim F_{A_t}]
        \ge \E\dim F_{A_t}-O\left(Nr(N^{-1}+\sqrt{Nt}+e^{-2N})\right).
\]
Substituting into~\eqref{eq:det1} and integrating gives
\[
        \E\log\det A_T
        \le -\frac12\int_0^T\E\dim F_{A_t}\,dt+O(1)
        +O\left(rT+N^{3/2}rT^{3/2}+NrTe^{-2N}\right).
\]
All three error terms are $o(1)$.  Indeed,
\[
        rT\le\frac{8\log(mNr)}{N}=o(1).
\]
Also,
\[
        N^{3/2}rT^{3/2}
        =O\left(\frac{(\log(mNr))^{3/2}}{\sqrt r}\right)=o(1)
\]
since $\log(mNr)=O(\log\varphi(m))$ and
$r=\lceil(\log\varphi(m))^4\rceil$, while
\[
        NrTe^{-2N}\le8\log(mNr)e^{-2N}=o(1).
\]
Absorbing these errors into the bounded term,
we obtain
\[
        \E\log\det A_T
        \le -\frac12\int_0^T\E\dim F_{A_t}\,dt+O(1).
\]
Finally, Lemma~\ref{lem:Fdimension} gives
\[
        \dim F_{A_t}\ge
        \frac{Nr}{2}-\frac{|\partial\calE_{A_t}\cap\Lambda|}{m}
\]
for every $t$.  Hence
\[
        -\frac12\int_0^T\E\dim F_{A_t}\,dt
        \le -\frac{NrT}{4}+\frac12\int_0^T\E
        \frac{|\partial\calE_{A_t}\cap\Lambda|}{m}\,dt.
\]
Substituting this bound into the preceding determinant estimate proves the proposition.
\end{proof}

\subsection{Fixed-vector contact estimates and orbit-contact localization}

The next estimate is the proof of~\cite[Proposition~4.1]{Klartag} with the
rank-one variance replaced by its invariant projection norm.
Recall from Lemma~\ref{lem:projection} that, for $x\ne0$,
\[
        \beta(x)^2
        =\frac{\|\Pi_K(x\otimes x)\|_{\HS}^2}{\|x\|^4}.
\]

\begin{lemma}[Fixed-vector contact estimate]\label{lem:fixed-contact}
Assume $a_0\Id$ is $\Lambda$-free.  For $0\ne x\in\Lambda$ and $t>0$,
\[
        \Prob(x\in\partial\calE_{A_t})
        \le
        2\Phi\left(
        \frac{a_0-\|x\|^{-2}}{\beta(x)\sqrt t}
        \right),
\]
where $\Phi$ is the standard Gaussian upper-tail function.
\end{lemma}

\begin{proof}
Set $M_s=A_s[x]-1$.  As in~\cite{Klartag}, this is a continuous martingale with
\[
        M_0=a_0\|x\|^2-1
        \ge0.
\]
Since $F_{A_s}\subseteq\mathcal Q_K$, its quadratic variation satisfies
\[
\begin{aligned}
        [M]_u
        &=\int_0^u\|\Pi_{A_s}(x\otimes x)\|_{\HS}^2\,ds \\
        &\le u\|\Pi_K(x\otimes x)\|_{\HS}^2
        =u\beta(x)^2\|x\|^4.
\end{aligned}
\]
The Dambis--Dubins--Schwarz and reflection-principle argument
in~\cite{Klartag} therefore gives
\[
        \Prob(x\in\partial\calE_{A_t})
        \le2\Phi\left(\frac{M_0}{\beta(x)\|x\|^2\sqrt t}\right)
        =2\Phi\left(
        \frac{a_0-\|x\|^{-2}}{\beta(x)\sqrt t}
        \right).\qedhere
\]
\end{proof}

Following the contact-localization step in~\cite[Proposition~4.2]{Klartag}, we now restrict
attention to the radial shell in which a new contact could occur on the
operator-norm good event.  For times with
$a_0-C_{\ref{lem:opnorm}}\sqrt{Nt}>0$, define
\[
        R_t=\left\{x\in V_\R:
        \frac1{a_0}\le\|x\|^2\le
        \frac1{a_0-C_{\ref{lem:opnorm}}\sqrt{Nt}}
        \right\}
\]
and define the orbit-contact functional
\begin{equation}
        K_t(\Lambda)=
        \frac1m\sum_{0\ne x\in\Lambda}\mathbf 1_{R_t}(x)
        \Phi\left(
        \frac{a_0-\|x\|^{-2}}{\beta(x)\sqrt t}
        \right).
\label{eq:Kt}
\end{equation}

The summand is $\mu_m$-invariant, so $K_t(\Lambda)$ is equivalently the
corresponding sum over nonzero $\mu_m$-orbits.

The next lemma is the contact-localization estimate
from~\cite{Klartag} with
the all-vector contact count replaced by an orbit count.

\begin{lemma}[Expected orbit contacts]\label{lem:contact-vs-K}
There is a universal constant $c_{\ref{lem:contact-vs-K}}>0$ such that, if
$a_0\Id$ is $\Lambda$-free, then for all $t>0$ in the range where
$a_0-C_{\ref{lem:opnorm}}\sqrt{Nt}>0$,
\[
        \E
        \frac{|\partial\calE_{A_t}\cap\Lambda|}{m}
        \le 2K_t(\Lambda)+e^{-c_{\ref{lem:contact-vs-K}}N}.
\]
\end{lemma}

\begin{proof}
On the good event
\[
        \|A_t-a_0\Id\|_{\op}\le C_{\ref{lem:opnorm}}\sqrt{Nt},
\]
any contact point $x$ at time $t$ satisfies $A_t[x]=1$, hence
\[
        (a_0-C_{\ref{lem:opnorm}}\sqrt{Nt})\|x\|^2\le1.
\]
Because $a_0\Id$ is $\Lambda$-free, every nonzero lattice point also satisfies $a_0\|x\|^2\ge1$.  Thus any contact point at time $t$ lies in the shell $R_t$.

Lemma~\ref{lem:fixed-contact} bounds the probability that a lattice vector in
$R_t$ is a contact point by twice the corresponding Gaussian tail.  If one
point in an orbit is a contact point, the whole orbit is, because both $A_t$
and $\Lambda$ are $\mu_m$-invariant.  Summing the vector bounds and dividing by
$m$ therefore gives the contribution $2K_t(\Lambda)$ on the good event.

On the bad event, the defining choice of $C_{\ref{lem:opnorm}}$ gives
probability at most $e^{-2N}$.  By Lemma~\ref{lem:standard-geometric}, the
number of boundary lattice points is at most $2(2^N-1)$, hence the number of
orbit contacts is at most this number.  The bad-event contribution is therefore
at most $2(2^N-1)e^{-2N}\le e^{-c_{\ref{lem:contact-vs-K}}N}$ after choosing
$c_{\ref{lem:contact-vs-K}}>0$ sufficiently small.
\end{proof}

\section{Averaging the orbit-contact functional}

We now average the orbit-contact functional $K_t$ defined in
\eqref{eq:Kt}.  Lemma~\ref{lem:contact-vs-K} reduces the expected number of
Brownian contact orbits to this deterministic lattice sum, so it remains to
average $K_t$ over the arithmetic lattice ensemble.

\subsection{Radial shell estimate}

We first prove the radial estimate needed for the averaged orbit-contact
bound.  Recall that $\beta(\omega)$ is the direction-dependent
projected-variance factor from Lemma~\ref{lem:projection}.  The calculation
below is the shell-integral calculation from~\cite[Lemma~4.3]{Klartag}, with
one modification: for a fixed direction, $\sqrt t$ is replaced by
$\beta(\omega)\sqrt t$ in the Gaussian tail.  The shell cutoff still comes
from the ambient operator-norm estimate and is unchanged.  We record the
short calculation to make this distinction explicit.

Recall the fixed choice $a_0=(1-1/N)^{-2}$ from
\eqref{eq:a0}.  Let $d\sigma$ denote probability measure on
$S^{N-1}$.  Our polar-coordinate normalization is
\[
        \frac1{\vol(B_N)}\int_{V_\R} f(x)\,dx
        =N\int_{S^{N-1}}\int_0^\infty f(\rho\omega)\rho^{N-1}\,d\rho\,d\sigma(\omega)
\]
for nonnegative Borel $f$.  Fix a direction $\omega\in S^{N-1}\subset V_\R$.
Recall that
\[
        \beta(\omega)^2
        =\frac{\|\Pi_K(\omega\otimes\omega)\|_{\HS}^2}{\|\omega\|^4}
        =\frac12\sum_{j=1}^s\|\omega_j\|^4,
\]
where $\omega=(\omega_1,\dots,\omega_s)\in\bigoplus_{j=1}^s\C^r$; see
Lemma~\ref{lem:projection}.  For $t>0$ with
$a_0-C_{\ref{lem:opnorm}}\sqrt{Nt}>0$, define
\begin{equation}
        J_t(\omega)
        =N\int_{1/\sqrt{a_0}}^{1/\sqrt{a_0-C_{\ref{lem:opnorm}}\sqrt{Nt}}}
        \Phi\left(
        \frac{a_0-\rho^{-2}}{\beta(\omega)\sqrt t}
        \right)
        \rho^{N-1}\,d\rho.
\label{eq:Jt}
\end{equation}
This is the radial contribution in polar coordinates normalized by
$\vol(B_N)$.

\begin{lemma}[Radial bound with angular variance]\label{lem:radial}
There are universal constants $c,C>0$ such that, whenever $t>0$ and $Nt\le c$,
\[
        J_t(\omega)
        \le \exp\left(
        \left(1+C\sqrt{Nt}\right)
        N^2\beta(\omega)^2t/8
        \right).
\]
\end{lemma}

\begin{proof}
Write $\beta=\beta(\omega)$.  By Lemma~\ref{lem:projection}, $\beta>0$.
The change of variables $y=(a_0-\rho^{-2})/(\beta\sqrt t)$ gives
\[
        J_t(\omega)
        =\frac{N\beta\sqrt t}{2}
        \int_0^{C_{\ref{lem:opnorm}}\sqrt N/\beta}
        \Phi(y)(a_0-\beta\sqrt t\,y)^{-(N+2)/2}\,dy.
\]
Choose $c>0$ small enough that
$C_{\ref{lem:opnorm}}\sqrt{Nt}\le1/2$ whenever $Nt\le c$.  On the interval of
integration,
\[
        0\le\frac{\beta\sqrt t\,y}{a_0}
        \le C_{\ref{lem:opnorm}}\sqrt{Nt}\le\frac12.
\]
Since $a_0\ge1$,
$(N+2)/(Na_0)=(1+2/N)(1-1/N)^2\le1$, and
$-\log(1-u)=u(1+O(u))$ uniformly for $0\le u\le1/2$, there is a universal
constant $C_1>0$ such that the choice
$b=(1+C_1\sqrt{Nt})N\beta\sqrt t/2$ gives
\[
        (a_0-\beta\sqrt t\,y)^{-(N+2)/2}\le e^{by}.
\]
After extending the positive integrand beyond its original upper endpoint,
integration by parts gives
\begin{align*}
        J_t(\omega)
        &\le b\int_0^\infty\Phi(y)e^{by}\,dy
        =\int_0^\infty \frac1{\sqrt{2\pi}}e^{-y^2/2}(e^{by}-1)\,dy
        \\
        &\le \exp\left(\frac{b^2}{2}\right)
        =\exp\left(
        \left(1+O(\sqrt{Nt})\right)\frac{N^2\beta^2t}{8}
        \right).\qedhere
\end{align*}
\end{proof}

\subsection{Angular block-mass estimate}\label{subsec:angular-block-mass}

We next average the direction-dependent bound from the preceding subsection.
This is the new analytic step: it uses concentration of the cyclotomic block
masses in place of the full rotational symmetry available in the original
argument.

\begin{lemma}[Angular block-mass moment]\label{lem:angular-moment}
There are universal constants $r_0,c,C>0$ such that, for all $r\ge r_0$ and all
$0\le\lambda\le c\sqrt r$, if
$\omega=(\omega_1,\dots,\omega_s)$ is a uniformly random unit vector in
$S^{N-1}\subset V_\R=\bigoplus_{j=1}^s\C^r$, with $\omega_j\in\C^r$, then
\[
        \E\exp\left(\lambda s\sum_{j=1}^s\|\omega_j\|^4\right)
        \le Ce^\lambda.
\]
\end{lemma}

\begin{proof}
Let $g=(g_1,\dots,g_s)$ be a standard Gaussian vector in
$V_\R=\bigoplus_{j=1}^s\C^r$.  Its direction
$\omega=g/\|g\|$ is uniformly distributed on $S^{N-1}$ and is independent of
$\|g\|$.  If $h$ is a standard Gaussian vector in $\R^d$, then expansion of
$\|h\|^4$ and the identities $\E h_i^2=1$ and $\E h_i^4=3$ give
\[
        \E\|h\|^4=d(d+2).
\]
Since $g_j$ has real dimension $2r$ and $g$ has real dimension $N=2sr$, it
follows that
\[
        \E\|\omega_j\|^4
        =\frac{\E\|g_j\|^4}{\E\|g\|^4}
        =\frac{2r(2r+2)}{N(N+2)}
        =\frac{r(r+1)}{sr(sr+1)}.
\]

Put
\[
        F(\omega)=\left(s\sum_{j=1}^s\|\omega_j\|^4\right)^{1/2}.
\]
By symmetry,
\[
        \E F(\omega)^2
        =s^2\frac{r(r+1)}{sr(sr+1)}
        =\frac{s(r+1)}{sr+1}
        \le1+\frac1r.
\]
Moreover, for $x,y\in S^{N-1}$,
\[
        |F(x)-F(y)|
        \le \sqrt{s}\left(\sum_{j=1}^s
        \bigl|\|x_j\|^2-\|y_j\|^2\bigr|^2\right)^{1/2}
        \le2\sqrt{s}\|x-y\|.
\]
Let $Y=(F(\omega)-\sqrt{1+1/r})_+$.  Since $N=2sr$, concentration on the sphere and
$\E F\le(\E F^2)^{1/2}\le\sqrt{1+1/r}$ imply that there are universal
constants $c,C>0$ such that
\[
        \Prob(Y>u)\le C e^{-2cru^2}
\]
for every $u\ge0$.
Therefore,
\[
        \E e^{crY^2}
        =1+\int_0^\infty 2cru e^{cru^2}\Prob(Y>u)\,du
        \le1+C\int_0^\infty 2cru e^{-cru^2}\,du=O(1).
\]
For $x,y\ge0$, AM--GM gives
\[
        2xy\le\frac{2}{\sqrt r}x^2+\frac{\sqrt r}{2}y^2.
\]
It follows that
\[
        (x+y)^2\le(1+2/\sqrt r)x^2+(1+\sqrt r/2)y^2.
\]
For $0\le\lambda\le c\sqrt r$, applying this with
$x=\sqrt{1+1/r}$ and $y=Y$ gives, after decreasing
$c>0$ if necessary and choosing $r_0$ sufficiently
large,
\[
\begin{aligned}
        \lambda F^2
        &\le\lambda\left(\sqrt{1+1/r}+Y\right)^2 \\
        &\le\lambda(1+2/\sqrt r)(1+1/r)
        +\lambda(1+\sqrt r/2)Y^2 \\
        &\le\lambda+O(1)+crY^2
        && \text{since $\lambda\le c\sqrt r$}.
\end{aligned}
\]
Together with the preceding moment bound, this yields
\[
        \E e^{\lambda s\sum_{j=1}^s\|\omega_j\|^4}
        =\E e^{\lambda F^2}
        \le e^{\lambda+O(1)}\E e^{crY^2}
        =O(e^\lambda).
\]
This proves the lemma.
\end{proof}

\begin{remark}
\label{re:effect_of_rank}
It is useful to note why the module rank $r$ is taken to grow.  The Gaussian
calculation in the proof gives
\[
        \E\left[s\sum_{j=1}^s\|\omega_j\|^4\right]
        =\frac{s(r+1)}{sr+1}=1+\frac1r+o(1).
\]
If $r$ were fixed, then Jensen's inequality shows that the exponential moment
arising in the angular average is at least
\[
        \exp\left((1+1/r+o(1))\lambda\right).
\]
Thus, even disregarding the radial error, the best contact bound of this form
that one could hope for is
\[
        \E_{\Lambda_g}K_t(\Lambda_g)
        =O\left(\frac1m
        \exp\left(\frac{(1+1/r)Nrt}{8}\right)\right).
\]
At the time $T$ chosen below, we have
$NrT/2=4\log(mNr)+O(1)$.  This bound would therefore yield only
\[
        \frac1m\int_0^T
        \exp\left(\frac{(1+1/r)Nrt}{8}\right)dt
        \asymp(mNr)^{1/r}.
\]
This is unbounded as $m\to\infty$.
For example, when $r=2$, the limiting coefficient is $3/2$ and the integral
grows like $(mNr)^{1/2}$.  To keep the integral bounded, one would instead have
to stop at a time satisfying
\[
        \frac{NrT}{2}\le\frac{4}{1+1/r}\log(mNr)+O(1).
\]
The determinant drift would then give only a volume gain of order
\[
        \exp\left(\frac{NrT}{8}\right)=O\left((mNr)^{r/(r+1)}\right),
\]
rather than order $mNr$.  Taking $r$ to grow makes the coefficient $1+1/r$
tend to $1$ and avoids this loss.
\end{remark}

We shall use the following immediate consequence in the time range needed
later.  It incorporates the radial error.

\begin{lemma}[Angular average in the admissible range]\label{lem:angular-radial}
There is a universal constant $C>0$ such that, for all sufficiently large
integers $m$, uniformly for
\[
        0<t\le\frac{8\log(mNr)}{Nr},
\]
if $\omega$ is a uniformly random unit vector in $S^{N-1}\subset V_\R$, then
\[
        \E J_t(\omega)
        \le C\exp\left(\frac{Nrt}{8}\right).
\]
\end{lemma}

\begin{proof}
Write $\omega=(\omega_1,\dots,\omega_s)$, where $\omega_j\in\C^r$.  By
Lemma~\ref{lem:projection} and $N=2sr$,
\[
        \frac{N^2\beta(\omega)^2t}{8}
        =\frac{Nrt}{8}s\sum_{j=1}^s\|\omega_j\|^4.
\]
Also,
\[
        Nt\le\frac{8\log(mNr)}r=o(1).
\]
Thus Lemma~\ref{lem:radial} applies for all sufficiently large $m$ and gives
\[
        J_t(\omega)\le \exp\left(
        \left(1+O(\sqrt{Nt})\right)
        \frac{Nrt}{8}s\sum_{j=1}^s\|\omega_j\|^4\right).
\]
Since $m\le2\varphi(m)^2$ for every positive integer $m$ and
$r=\lceil(\log\varphi(m))^4\rceil$, we have
\[
        \log(mNr)=O(\log\varphi(m))=o(\sqrt r).
\]
The coefficient of the block-mass sum in the preceding exponential is
therefore $o(\sqrt r)$, so Lemma~\ref{lem:angular-moment} applies and gives
\[
        \E J_t(\omega)
        =O\left(\exp\left(
        \left(1+O(\sqrt{Nt})\right)\frac{Nrt}{8}\right)\right)
        =O\left(\exp\left(\frac{Nrt}{8}\right)\right).
\]
For the last estimate, the error in the exponent is uniformly
$O((\log(mNr))^{3/2}/\sqrt r)=o(1)$.  This proves the lemma.
\end{proof}

\subsection{Averaged contact bound}

\begin{proposition}[Averaged contact estimate]\label{prop:avgK}
There is a universal constant $C_{\ref{prop:avgK}}>0$ such that, for all
sufficiently large integers $m$, uniformly for
\[
        0<t\le\frac{8\log(mNr)}{Nr},
\]
we have
\[
        \E_{\Lambda_g}K_t(\Lambda_g)
        \le \frac{C_{\ref{prop:avgK}}}{m}\exp\left(\frac{Nrt}{8}\right).
\]
\end{proposition}

\begin{proof}
Apply Theorem~\ref{thm:siegel} to the nonnegative Borel function given by the
summand in~\eqref{eq:Kt} on $R_t$ and zero outside $R_t$.  By the
polar-coordinate formula above,
\[
        \E_{\Lambda_g}K_t(\Lambda_g)
        =\frac1m\E J_t(\omega).
\]
Lemma~\ref{lem:angular-radial} proves the proposition.
\end{proof}

This proposition is the replacement for the averaged contact estimate
in~\cite{Klartag}.  The structure is
the same--a fixed-vector contact estimate
followed by a Siegel mean formula and a radial shell integral--but the result
has two new features: the all-vector sum is normalized by the orbit size,
producing the factor $1/m$, and the angular block-mass estimate changes the
exponent to $Nrt/8$.

\section{Selecting a stopping time and a good lattice}

Proposition~\ref{prop:avgK} suggests choosing the stopping time so that its
averaged contact bound remains small while $NrT$ is as large as possible.
Increase $C_{\ref{prop:avgK}}$ if necessary so that
$C_{\ref{prop:avgK}}\ge1$, and set
\begin{equation}
        T=\frac8{Nr}\log\left(\frac{mNr}{800C_{\ref{prop:avgK}}}\right).
\label{eq:chosen-time}
\end{equation}
For all sufficiently large integers $m$, this time is positive and lies in
the range of Proposition~\ref{prop:avgK}.  Moreover,
\[
        \frac1m\int_0^T\exp(Nrt/8)\,dt
        \le\frac1{100C_{\ref{prop:avgK}}}.
\]
We now select a lattice for which the corresponding integral is small.  Apart
from the arithmetic lattice ensemble, this is the same selection argument
used by Klartag~\cite[proof of Proposition~5.1]{Klartag}.

\begin{proposition}[Good lattice]\label{prop:good-lattice}
For all sufficiently large integers $m$, there exists a $\mu_m$-stable
lattice $\Lambda\subset V_\R$ of covolume $\vol(B_N)$ such that $a_0\Id$ is
$\Lambda$-free and
\[
        \int_0^T K_t(\Lambda)\,dt\le\frac1{10}.
\]
\end{proposition}

\begin{proof}
Proposition~\ref{prop:avgK} gives
\[
        \E_{\Lambda_g}K_t(\Lambda_g)
        \le \frac{C_{\ref{prop:avgK}}}{m}\exp(Nrt/8).
\]
This holds for every $0<t\le T$.
Extend $K_t$ by zero at $t=0$.  Since $\Lambda_g=g\OK^r$, the map
$(g,t)\mapsto K_t(\Lambda_g)$ is nonnegative and Borel: locally after lifting
from the quotient to $\operatorname{SL}_r(K_\infty)$, it is a
countable sum of nonnegative Borel functions indexed by
$\OK^r\setminus\{0\}$.  Hence Tonelli's theorem gives
\[
        \E_{\Lambda_g}\int_0^T K_t(\Lambda_g)\,dt
        =\int_0^T\E_{\Lambda_g}K_t(\Lambda_g)\,dt
        \le \frac{C_{\ref{prop:avgK}}}{m}\int_0^T
        \exp\left(\frac{Nrt}{8}\right)\,dt
        \le \frac{C_{\ref{prop:avgK}}}{m}\cdot\frac8{Nr}
        \exp\left(\frac{NrT}{8}\right)=\frac1{100}.
\]
Thus
\begin{equation}
        \E_{\Lambda_g}\int_0^T K_t(\Lambda_g)\,dt\le\frac1{100}.
\label{eq:Ksmallavg}
\end{equation}

By Theorem~\ref{thm:siegel},
\[
        \E_{\Lambda_g}
        \#\{0\ne x\in\Lambda_g:\|x\|\le1-1/N\}
        =(1-1/N)^N<e^{-1}.
\]
Thus the probability that the closed ball of radius $1-1/N$ contains a nonzero
lattice point is at most $e^{-1}$.  By Markov's inequality
and~\eqref{eq:Ksmallavg},
\[
        \Prob\left(\int_0^T K_t(\Lambda_g)\,dt>\frac1{10}\right)
        \le\frac1{10}.
\]
Since $e^{-1}+1/10<1$, there exists a $\mu_m$-stable lattice $\Lambda$ of
covolume $\vol(B_N)$ such that
\[
        \{0\ne x\in\Lambda:\|x\|\le1-1/N\}=\varnothing
\]
and
\begin{equation}
        \int_0^T K_t(\Lambda)\,dt\le\frac1{10}.
\label{eq:goodK}
\end{equation}
For this lattice, the ball $\{x:\|x\|<1-1/N\}$ is $\Lambda$-free.  Since
$a_0=(1-1/N)^{-2}$, this ball is $\calE_{a_0\Id}$, so $a_0\Id$ is
$\Lambda$-free.

\end{proof}

We now complete the proof of the main technical result,
Theorem~\ref{thm:arithmetic-lattice}.

\begin{proof}[Proof of Theorem~\ref{thm:arithmetic-lattice}]
Let $\Lambda$ be the lattice selected in Proposition~\ref{prop:good-lattice},
and run the process of Proposition~\ref{prop:process} from $A_0=a_0\Id$ up to
the time $T$ defined in~\eqref{eq:chosen-time}.  We have
\[
        NT=\frac8r\log\left(\frac{mNr}{800C_{\ref{prop:avgK}}}\right)=o(1).
\]
Thus $a_0-C_{\ref{lem:opnorm}}\sqrt{NT}>0$ for all sufficiently large $m$, so the
shell $R_t$ and Lemma~\ref{lem:contact-vs-K} are valid for every $0<t\le T$;
the endpoint $t=0$ is irrelevant for the time integral.

By Lemma~\ref{lem:contact-vs-K} and~\eqref{eq:goodK},
\begin{equation}
        \int_0^T\E
        \frac{|\partial\calE_{A_t}\cap\Lambda|}{m}\,dt
        \le 2\int_0^T K_t(\Lambda)\,dt
        +Te^{-c_{\ref{lem:contact-vs-K}}N}
        =O(1).
\label{eq:Cint}
\end{equation}

Using~\eqref{eq:Cint} in Proposition~\ref{prop:det-drift}, we obtain, for all
sufficiently large $m$,
\[
        \E\log\det A_T
        \le -\frac{NrT}{4}+O(1).
\]
Since
\[
        \frac{NrT}{2}=4\log\left(\frac{mNr}{800C_{\ref{prop:avgK}}}\right)
        =4\log(mNr)+O(1),
\]
we have
\[
        \E\log\det A_T
        \le -2\log(mNr)+O(1).
\]
The random variable $\log\det A_T$ is integrable by the argument in
Proposition~\ref{prop:det-drift}.  Hence at least one Brownian realization
satisfies
\[
        \log\det A_T\le -2\log(mNr)+O(1),
\]
and hence
\begin{equation}
        \det A_T=O\left((mNr)^{-2}\right).
\label{eq:detfinal}
\end{equation}
The ellipsoid
\[
        E=\calE_{A_T}=\{x:A_T[x]<1\}
\]
satisfies $E\cap\Lambda=\{0\}$ by construction.  The final conversion from a
determinant bound to a large lattice-free ellipsoid is immediate.  Its
volume is
\[
        \vol(E)=\det(A_T)^{-1/2}\vol(B_N).
\]
By~\eqref{eq:detfinal},
\[
        \vol(E)
        \ge cmNr\vol(B_N).
\]
This proves the theorem.
\end{proof}

\paragraph*{Acknowledgements}

N.G.\ acknowledges support from the Swiss National Science Foundation grant 225437.

\medskip
\noindent
Elisha B. Abuya\\
Tel Aviv University, Israel.\\
\textit{e-mail:} \href{mailto:ebabuya@tauex.tau.ac.il}{\texttt{ebabuya@tauex.tau.ac.il}}

\medskip
\noindent
Nihar Gargava\\
Universit\'e Paris-Saclay, France.\\
\textit{e-mail:} \href{mailto:nihar.gargava@universite-paris-saclay.fr}{\texttt{nihar.gargava@universite-paris-saclay.fr}}

\medskip
\noindent
Yufei Zhao\\
Massachusetts Institute of Technology, USA.\\
\textit{e-mail:} \href{mailto:yufeiz@mit.edu}{\texttt{yufeiz@mit.edu}}

\end{document}